\documentclass[12pt,twoside]{article}
\usepackage{amssymb,amsthm}
\usepackage[namelimits,sumlimits,fleqn]{amsmath}
\usepackage{setspace}
\usepackage[ hmargin={2.2cm,2.4cm}, top=36mm,
    a4paper,pdftex, textheight=230mm,
    headheight=20mm, headsep=10mm, bindingoffset=0.9cm]
    {geometry}
\usepackage[small, margin=20pt]{caption}
\usepackage[usenames,dvipsnames]{color}
\definecolor{darkblue}{rgb}{0.0,0.0,0.6}
\usepackage{float}\restylefloat{figure}
\usepackage{url}
\usepackage[backref=page]{hyperref}
\hypersetup{colorlinks=true,breaklinks=true,
            linkcolor=darkblue,urlcolor=darkblue,
            anchorcolor=darkblue,citecolor=darkblue}

\usepackage{fancyhdr}
\allowdisplaybreaks[4]

\title{Pinned algebraic distances determined by Cartesian products in $\mathbb{F}_p^2$}
\author{Giorgis Petridis}
\date{}

\theoremstyle{plain}
\newtheorem{theorem}{Theorem}[section]
\newtheorem{lemma}[theorem]{Lemma}

\theoremstyle{definition}
\newtheorem*{acknowledgement}{Acknowledgement}
\newtheorem*{notation}{Notation}
\newtheorem*{claim}{Claim}

\renewcommand*{\backref}[1]{}
\renewcommand*{\backrefalt}[4]{%
    \ifcase #1 (Not cited.)%
    \or        (p.~\!#2)%
    \else      (pp.~\!#2)%
    \fi}

\newcommand{\F}{\mathbb{F}_p} 
\newcommand{\Fq}{\mathbb{F}_q} 
\newcommand{\R}{\mathbb{R}} 
\newcommand{\ds}{\displaystyle} 

\begin{document}

\onehalfspacing

\pagenumbering{arabic}

\setcounter{section}{0}

\bibliographystyle{plain}

\maketitle

\begin{abstract}
Let $p$ be an odd prime and $A \subseteq \mathbb{F}_p$ be a subset of the finite field with $p$ elements. We show that $A \times A \subseteq \mathbb{F}_p^2$ determines at least a constant multiple of $\min\{p, |A|^{3/2}\}$ distinct pinned algebraic distances.
\end{abstract}

\section[Introduction]{Introduction}
\label{Intro}

\let\thefootnote\relax\footnotetext{The author is supported by the NSF DMS Grant 1500984.}

Erd\H{o}s proved in~\cite{Erdos1946} that a finite planar set $E \subset \R^2$ determines at least $\Omega(|E|^{1/2})$ distinct distances and conjectured it determines at least $|E|^{1-o(1)}$. 
The conjecture was proved by Guth and Katz after decades of partial progress in their seminal paper~\cite{Guth-Katz2015}. The corresponding conjecture for distances ``pinned'' at some point of $E$ remains open, with the best known bound being due to Katz and Tardos~\cite{Katz-Tardos2004}.

The question has been studied in the context of two-dimensional vector spaces over finite fields, where the algebraic distance between two points $u=(u_1,u_2)$ and $v=(v_1,v_2)$ is defined as
\[
\| u- v \| = (u_1-v_1)^2 + (u_2-v_2)^2.
\]
For notational brevity we denote the set of algebraic distances determined by $E$ by
\[
\Delta(E) = \{ \| u - v\| : u,v \in E\}
\]
and the set of algebraic distances determined by $E$ pinned at some $u \in E$ by
\[
\Delta_u(E) = \{ \| u - v\|  : v \in E\}.
\]

The algebraic structure of vector spaces $\Fq^2$ is very different than that of $\R^2$ and so peculiarities arise. For example, if $-1$ is a square in the field, say $i^2=-1$, then the \emph{isotropic line} $\{ (t, it) : t \in \Fq\}$ determines only one distance: 0. Subfields also pose obstructions to improving upon an Erd\H{o}s type lower bound $|\Delta(E)| = \Omega(|E|^{1/2})$, which can be obtained once $E$ is not contained in an isotropic line: If $E$ is the Cartesian product of a subfield, then $\Delta(E)$ is precisely the subfield and so $|\Delta(E)| = |E|^{1/2}$. 

Bourgain, Katz, and Tao worked over prime order fields $\F$ where $p$ is an odd prime congruent to $3 \pmod 4$ (and so $-1$ is not a quadratic residue) in their ground breaking paper~\cite[Theorem 7.1]{BKT2004}. In this setting the above obstructions do not occur and Bourgain, Katz, and Tao proved that for all $\delta>0$, there exists $c>0$ such that for all sets $E \subset \F^2$ that satisfy $|E| \leq p^{2-\delta}$ there exists $u \in E$ such that
\[
|\Delta_u(E)| \geq |E|^{\tfrac{1}{2} +  c}.
\]
Bourgain, Katz, and Tao credit their argument to Chung, Szemer\'edi, and Trotter (see the first line of the proof of Theorem 7.1 in~\cite{BKT2004}; and the last line of p.~\!1 in~\cite{CST1992}) and is based on a point-line incidence theorem. An explicit $c$ can be obtained by applying more recent point-line incidence results found in~\cite{AYMRS,Helfgott-Rudnev2011,Jones2016,SdZ}. The Stevens and de Zeeuw point-line incidence bound~\cite{SdZ} implies that, if $E$ is not contained in an isotropic line and satisfies the hypothesis $|E| = O(p^{15/11})$, then there exists $u \in E$ such that
\[
|\Delta_u(E)| = \Omega\left( |E|^{\tfrac{1}{2}+ \tfrac{1}{30}}\right).
\] 

For Cartesian products in $\F^2$ (prime $p$), Aksoy Yazici, Murphy, Rudnev, and Shkredov proved in~\cite{AYMRS} that $E = A \times A$ determines $\Omega(|E|^{\tfrac{1}{2} +  \tfrac{1}{16}})$ algebraic distances if $|E| = O(p^{16/15})$. It is natural to also express their result as $ |\Delta(A \times A)| = \Omega(|A|^{9/8}).$ Inserting the point-line incidence result of Stevens and de Zeeuw~\cite{SdZ} in the argument of Bourgain, Katz and Tao yields that there exists $u \in A \times A$ such that
\[
|\Delta_u(A \times A)| = \Omega\left(|A \times A|^{5/8}\right) = \Omega\left(|A|^{5/4}\right),
\]
under the condition $|A \times A| = O(p^{4/3})$. This is the best known bound in the literature. 

The above condition $|A \times A| = O(p^{4/3})$ is not restrictive. To see why one must  investigate the complementary question of determining a lower bound on $|E|$ so that $\Delta(E)$ is about as large as it can ever be, which was first studied by Iosevich and Rudnev in~\cite{Iosevich-Rudnev2007}. Iosevich and Rudnev showed that $|E| > 4 q^{3/2}$ implies $\Delta(E) = \Fq$. Chapman, Erdo{\u{g}}an, Hart, Iosevich and Koh proved that $|E| > q^{4/3}$ implies that $\Delta(E)$ contains a positive proportion of the elements of $\Fq$~\cite{CEHIK2012}. In the same paper, it was shown that if $E = A \times A$ is a Cartesian product and $|E|> q^{4/3}$, then there exists $u \in E$ such that $|\Delta_u(E)| > q/3$. Hanson, Lund and Roche-Newton strengthened this result in~\cite{HLRN2016} by establishing a similar result for all sets in $\Fq^2$: If $|E| > q^{4/3}$, then there exists $u \in E$ such that the pinned distance set $\Delta_u(E)$ contains a positive proportion of the elements of $\Fq$. Corresponding questions in $\Fq^d$ for $d\geq 3$ were  studied in~\cite{HIKR2011}. 
 
We prove a result on pinned algebraic distances for Cartesian products in $\F^2$, which goes beyond what can be achieved by current knowledge on point-line incidences and complements that of Chapman, Erdo{\u{g}}an, Hart, Iosevich and Koh.

\begin{theorem}\label{Dist}
Let $p$ be an odd prime and $A \subseteq \F$. There exist $a,b \in A$ such that 
\[
|\Delta_{(a,b)}(A \times A)| = \Omega(\min\{p, |A|^{3/2}\}).
\]
\end{theorem}

Note that $ \Delta_{(a,b)}(A \times A) = \{ (a-c)^2 + (b-d)^2 : c,d \in A\}  =  (A-a)^2 + (A-b)^2$ and that, in the notation used above, the theorem states there exists $u \in E = A \times A$ such that $|\Delta_u(E)| = \Omega(\min\{p, |E|^{3/4}\}).$

Our proof is based on a simple averaging argument found in the paper of Bourgain, Katz, and Tao~\cite{BKT2004} and a point-plane incidence theorem of Rudnev from~\cite{Rudnev}. Our method, therefore, can be partly traced to the work of Guth and Katz (c.f. Section~\ref{RudSec}). It is possible to put the proof, which is reminiscent of an argument in~\cite{GPInc}, in the context of the image set theorem of Aksoy Yazici, Murphy, Rudnev, and Shkredov~\cite[Theorem 1]{AYMRS}, but we opted for a more direct albeit longer presentation. 
\begin{acknowledgement}
The author would like to thank Alex Iosevich and Misha Rudnev for generously sharing their insight; and the referee for a very careful reading of the paper and an insightful report.
\end{acknowledgement}
 
\begin{notation}
We use Landau's notation so that both statements $f = O(g)$ and $g = \Omega(f)$ mean there exists an absolute constant $C$ such that $f \leq C g$ and $f = \Theta(g)$ stands for $f=O(g)$ and $f= \Omega(g)$. The letter $p$ denotes a prime, $q$ a prime power, $\Fq$ the finite field with $q$ elements and $\Fq^d$ the $d$-dimensional vector space over $\Fq$. An isotropic line is a line such that any two points on that line are at distance $0$ from each other. 
\end{notation}

\section[Rudnev's Theorem]{Rudnev's point-plane incidences theorem}
\label{RudSec}

Given a point set $P$ and a collection of planes $\Pi$ in $\F^3$, a point-plane incidence is an ordered pair $(u, \pi) \in P \times \Pi$ such that $u \in \pi$. Theorem~\ref{Dist} is a consequence of the following point-plane incidence bound of Rudnev~\cite{Rudnev}. We state the theorem in the simplest form adequate for our purpose.
\begin{theorem}[Rudnev]\label{Rudnev}
Let $p$ be an odd prime, $P$ be a set of points in $\F^3$ and $\Pi$ be a set of planes in $\F^3$. Suppose that $|P| = |\Pi|= O(p^2)$ and denote by $k$ the maximum number of collinear points in $P$. The number of point-plane incidences is $O(|P|^{3/2} + k |P|)$.
\end{theorem}
The proof of Rudnev's theorem has its roots in the solution to the Erd\H{o}s distinct distance problem for sets in $\R^2$ by Guth and Katz~\cite{Guth-Katz2015} and the Klein--Pl\"ucker line geometry formalism~\cite[Chapter~2]{Pottmann-Wallner2001}. So it depends on classical techniques such as the polynomial method (see~\cite{Guth2016}), and properties of ruled surfaces (see~\cite{Katz2014}) and the Klein quadric (see~\cite{RudnevSelig2016}). Applying Theorem~\ref{Rudnev} in the setting of Theorem~\ref{Dist} is similar to how Theorem~\ref{Rudnev} was applied by Aksoy Yazici, Murphy, Rudnev, and Shkredov in~\cite{AYMRS}
.

\section[An averaging argument]{An averaging argument}

The second ingredient in the proof of Theorem~\ref{Dist} is a simple observation that can indirectly be traced back at least to a paper of Chung, Szemer\'edi, and Trotter~\cite{CST1992} and which was first applied in the finite field context by Bourgain, Katz and Tao~\cite{BKT2004}. We state and prove it for sets in $\Fq^2$, though we only apply it in $\F^2$.

\begin{lemma}\label{CST}
Let $E \subseteq \Fq^2$ and $N$ be the number of solutions to 
\[
 \| u-v\| = \|u-w\| \text{ with } u,v,w \in E.
\] 
There exists $u \in E$ such that $\displaystyle |\Delta_u(E)| \geq \frac{|E|^3}{N}$.
\end{lemma}

\begin{proof}
We begin by getting an alternative expression for $N$. We denote by $1_\mathcal{E}$ the indicator function for the event $\mathcal{E}$.
\begin{align*}
N 
& = \sum_{u, v, w \in E} 1_{\| u-v\| = \|u-w\|} \\
& =  \sum_{u, v, w \in E} \sum_{x \in \Fq} 1_{\| u-v\| = x = \|u-w\|} \\ 
& = \sum_{u \in E} \sum_{x \in \Fq} \sum_{v, w \in E} 1_{\| u-v\| = x = \|u-w\|} \\ 
& =\sum_{u \in E} \sum_{x \in \Fq} \left( \sum_{v \in E} 1_{\{\| u-v\| = x}\right)^2.
\end{align*}
Therefore there exists $u \in E$ such that
\[
\sum_{x \in \Fq} \left( \sum_{v \in E} 1_{\| u-v\| = x} \right)^2 \leq \frac{N}{|E|}.
\]
By the Cauchy-Schwarz inequality:
\[
|\Delta_u(E)| \geq \frac{\left( \sum_{x \in \Fq}  \sum_{v \in E} 1_{\| u-v\| = x} \right)^2}{\sum_{x \in \Fq} \left( \sum_{v \in E} 1_{\| u-v\| = x} \right)^2} \geq \frac{|E|^2}{N/|E|} = \frac{|E|^3}{N}.
\]
\end{proof} 
Note here that if $E$ is an isotropic line, then $N = |E|^3$; and if $E = F \times F$ is the Cartesian product of a subfield $F$, then $N = |F|^{5/2}$.

Bourgain, Katz and Tao~\cite[Theorem 7.1]{BKT2004}, in an argument they credit to Chung, Szemer\'edi, and Trotter~\cite{CST1992}, note that for every pair of distinct $v,w \in E$ that do not belong to an isotropic line, every $u$ such that $(u,v,w)$ contributes 1 to $N$ belongs to the perpendicular bisector of $v$ and $w$. Therefore, when there are no isotropic lines, for each fixed $w$ the number of $(u,v)$ such that $(u,v,w)$ contributes 1 to $N$ equals the number of incidences between $E\setminus\{w\}$ and a family of $|E|-1$ lines. This means that improving the state-of-the art on point-line incidence bounds gives improved bounds for $N$ and consequently for the pinned distance set. 

It should be noted here~\cite[last line on p.~\!1]{CST1992} that Theorem~\ref{Dist} (and in fact the stronger analogous statement for all point sets and not just Cartesian products) would follow from the argument of Bourgain, Katz, and Tao if there was in our disposal a point-line incidence bound comparable to what is known in $\mathbb{R}^2$~\cite{Szemeredi-Trotter1983}.

For Cartesian products we follow a different approach. We are able to avoid treating each $w$ separately by reducing the question to one about point-plane incidences in $\F^3$.

\section[Proof of Theorem 1.1]{Proof of Theorem~\ref{Dist}}

Recall that Theorem~\ref{Dist} states that for all $A \subseteq \F$ there exist $a,b \in A$ such that 
\[
|\Delta_{(a,b)}(A \times A)| = \Omega(\min\{p, |A|^{3/2}\}).
\]
Lemma~\ref{CST} reduces Theorem~\ref{Dist} to proving that $N = O(|A|^{9/2} + |A|^6 /p)$, a bound in line with other applications of Rudnev's theorem~\cite{AYMRS,GPProdDiff,RRS2016,RSS}. As we will see bellow, the condition $|P| = O(p^2)$ in Theorem~\ref{Rudnev} forces us to require $|A| = O(p^{2/3})$. For this technical reason, we prove the following claim.
\begin{claim}
Suppose that $A \subseteq \F$ satisfies $|A| = O(p^{2/3})$. There exists $u \in A \times A$ such that $|\Delta_u(A \times A)| = \Omega(|A|^{3/2})$.
\end{claim}

The claim implies the theorem because, if $A \subseteq \F$ has $|A| = \Omega(p^{2/3})$, then we pass to a subset $A' \subseteq A$ with $|A'|= \Theta(p^{2/3})$. By the claim, there exists $u \in A'$ such that
\[
|\Delta_u(A \times A)| \geq |\Delta_u(A' \times A')| = \Omega(|A'|^{3/2}) =  \Omega(p).
\] 

Our first task in proving the claim is to express $N$ using the coordinates of $u,v,w$, noting that the coordinates are elements of $A$. By standard properties of inner product $N$ is the number of solutions to
\[
2 u_1 (v_1-w_1) + 2 u_2 (v_2 - w_2) + (v_2^2-w_2^2) = (v_1^2-w_1^2)\;~ u_i, v_i, w_i \in A.
\]
We treat differently the case where $v_1 = w_1$ or $v_2=w_2$. The number of solutions where $v_1=w_1 =x$ for any given $x \in A$. We must have $v_2 = w_2$ or $2u_2 = - v_2 - w_2$. In either case there are at most $|A|^3$ solutions. Treating the case where $v_2 = w_2$ identically, we see there are at most $4 |A|^4$ solutions with $v_1=w_1$ or $v_2 = w_2$. 

Assuming from now on that $v_1 \neq w_1$ and $v_2 \neq w_2$ we express the above equation as
\[
(2 u_1,  v_2 - w_2 , v_2^2-w_2^2) \cdot (v_1-w_1 , 2 u_2 ,1) = v_1^2-w_1^2, \; v_i, w_i \in A, \, v_i \neq w_i.
\]
This reduces the question to a point-plane incidence bound for which we apply Theorem~\ref{Rudnev}. The details are as follows.

\emph{First step}: Define a set of distinct points
\[
P = \{ (2 u_1,  v_2 - w_2 , v_2^2-w_2^2) : u_1, v_2, w_2 \in A\}
\] 
and a family of distinct planes 
\[
\Pi = \{ \{ x \in \Fq^2: x \cdot (v_1-w_1 , 2 u_2 ,1) = v_1^2-w_1^2\}  : u_2, v_1, w_1 \in A\}.
\]
The points and planes are distinct because, for distinct $\alpha, \beta \in \F$, the ordered pair $(\alpha - \beta, \alpha^2 - \beta^2)$ determines uniquely the ordered pair $(\alpha - \beta, \alpha + \beta)$. Since the characteristic is odd, the ordered pair $(\alpha - \beta, \alpha + \beta)$ determines uniquely the ordered pair $(\alpha, \beta)$. Moreover, since the characteristic is odd, $2 \alpha$ uniquely determines $\alpha$. 

Our choice of $P$ and $\Pi$ ensures that the number of point-plane incidences between $P$ and $\Pi$ is precisely the contribution to $N$ coming from $v_i \neq w_i$. 

\emph{Second step}: The cardinalities of $P$ and $\Pi$ satisfy $|P|=|\Pi| = |A|^3 -2|A|^2 +|A| \leq |A|^3$.

\emph{Third step}: Our assumption that $|A| = O(p^{2/3})$ implies that $|P| \leq |A|^3  = O(p^2)$.

\emph{Fourth step}: The maximum number $k$ of collinear points in $P$ is at most $2|A|$. Recall that the elements of $P$ are of the form $(2 u_1,  v_2 - w_2 , v_2^2-w_2^2)$ with $u_1,v_2,w_2 \in A$, $v_2 \neq w_2$. To prove that the maximum number $k$ of collinear points in $P$ is at most $2|A|$ we consider two different cases. 

Lines not contained in any plane of the form \{$X = $ constant\} intersect each plane $\{ X = 2u\}$ in at most one point. So they are incident to at most $|A|$ points from $P$.

For lines contained in a plane of the form $\{X = 2u\}$  we seek to bound the maximum number of collinear points of the form $(v_2 - w_2 , v_2^2-w_2^2) \in \F^2$ with $v_2,w_2 \in A$. We distinguish between three types of lines: $\{Y = $ constant\}, $\{Z = $ constant\} and $\{Z = mY + b\}$. Lines of the form $\{ Z = \kappa\}$ are incident to at most $2|A|$ elements of the form $(v_2 - w_2 , v_2^2-w_2^2)$ with $v_2,w_2 \in A$, because for each $v_2$ there are at most two $w_2 \in \F$ such that $v_2^2  - w_2^2 = \kappa$. Similarly, lines of the form $\{ Y = \kappa\}$ are incident to at most $|A|$ elements of the form $(v_2 - w_2 , v_2^2-w_2^2)$ with $v_2,w_2 \in A$. Finally, lines of the form $\{ Z = m Y + b\}$ are incident to at most $2|A|$ elements of the form $(v_2 - w_2 , v_2^2-w_2^2)$ with $v_2,w_2 \in A$, because $v_2$ and $w_2$ must satisfy $ v_2^2 - w_2^2 = m (v_2 - w_2) + b$ and for each $v_2$ there are at most two $w_2 \in \F$ that satisfy the resulting quadratic equation.

\emph{Fifth step}: Theorem~\ref{Rudnev} implies that $N$, which equals the number of point-plane incidences between $P$ and $\Pi$ plus $O(|A|^4)$, satisfies $\ds N = O(|A|^{9/2} + |A|^4) = O(|A|^{9/2}).$

\emph{Sixth step}: Lemma~\ref{CST} implies that there exists $u \in A \times A$ such that
\[
|\Delta_u(A \times A)| \geq \frac{|A \times A|^3}{N} = \Omega\left( \frac{|A|^6}{|A|^{9/2}}\right) = \Omega(|A|^{3/2}) =  \Omega(|A \times A|^{3/4}).
\]

\phantomsection

\addcontentsline{toc}{section}{References}

\bibliography{all}

\vskip 0.5cm

\hspace{20pt} Department of Mathematics, University of Georgia, Athens, GA, USA.

\hspace{20pt} \textit{Email address}: \href{mailto:giorgis@cantab.net}{giorgis@cantab.net}

\end{document}